\documentclass[12pt]{amsart}
\usepackage{amsmath,amssymb,amsthm,graphicx,comment}
\usepackage[T2A]{fontenc}
\usepackage[russian, english]{babel} 
\usepackage[left=2.5cm,right=2.2cm,top=2.5cm,bottom=3.5cm]{geometry}

\let\le\leqslant
\let\ge\geqslant
\let\leq\leqslant
\let\geq\geqslant

\newtheorem{theorem}{Theorem}

\newtheorem{lemma}{Lemma}
\newtheorem*{remark*}{Remark}

\theoremstyle{definition}
\newtheorem{definition}{Definition}
\newtheorem*{hypothesis*}{Hypothesis}
\begin{document}
\thispagestyle{empty}
\begin{center}
\textbf{\large Maximum Spread of Vertex Degrees in a Simple Graph}\\
\vfil
Sergey Dmitrievich Onishchenko\footnote{Saint Petersburg University, 7/9 Universitetskaya nab., St. Petersburg, 199034 Russia. E-mail: dimer03@mail.ru}\\
\vfil
August 16, 2025
\vfil \vfil
\end{center}

\section*{Abstract}
We consider the following problem: let $n>k$ be natural numbers, and let $G$ be a graph on $n$ vertices (undirected, without loops or multiple edges).
Denote by $h_k(G)$ the number of unordered pairs of vertices in the graph $G$ whose degrees differ by less than $k$. We aim to determine the smallest possible value $f(n,k)$ of the quantity $h_k(G)$.
Interest in this question is motivated by the fact that the bipartite analogue of the problem enabled S. Cichomski and F. Petrov \cite{CP} to prove the Burdzy–Pitman conjecture on the spread of independent coherent random variables.

The problem has been solved under a number of restrictions on $n$ and $k$.
A conjecture about the answer in the general case is also presented.

\section*{Acknowledgement}
The work of S. Onishchenko was performed at the Saint Petersburg Leonhard Euler International Mathematical Institute and supported by the Ministry of Science and Higher Education of the Russian Federation (agreement no. 075–15–2025–343).

\section*{Introduction}
We consider the following problem: let $n>k$ be natural numbers, and let $G$ be a graph on $n$ vertices (undirected, without loops or multiple edges).
Denote by $h_k(G)$ the number of unordered pairs of vertices in the graph $G$ whose degrees differ by less than $k$. We aim to determine the smallest possible value $f(n,k)$ of the quantity $h_k(G)$.
Interest in this question is motivated by the fact that the bipartite analogue of the problem enabled S. Cichomski and F. Petrov \cite{CP} to prove the Burdzy–Pitman conjecture on the spread of independent coherent random variables.

The problem has been solved under a number of restrictions on $n$ and $k$.
A conjecture about the answer in the general case is also presented. If a counterexample to this conjecture exists for some $n$ and $k$, it can be proven that for these $n$ and $k$ there also exists a counterexample of the following form:

The vertices are partitioned into sets $B_1, B_2, B_3, B_4, B_5$. $B_1, B_2$ are cliques; $B_4, B_5$
are anticliques (independent sets); the bipartite graphs
induced on the pairs of sets $(B_1,B_2)$, $(B_1,B_3)$, $(B_1,B_4)$, $(B_2,B_3)$ are
complete; on the pairs of sets $(B_3,B_5)$, $(B_2,B_5)$, $(B_3,B_4)$ they are empty.

Unfortunately, this statement did not allow us to extend the set of pairs $n$ and $k$ for which the problem is solved, so its proof will not be provided here.

We will need a useful lemma from \cite{E}:

\begin{lemma}
In any graph $G$ on $n>k$ vertices, one can find $k+1$ vertices whose degrees
differ pairwise by less than $k$.
\end{lemma}

For the sake of completeness, we provide (a simpler proof than in \cite{E})

\begin{proof} Label the vertices from 1 to $n$
in order of increasing degrees
$d_1\leq \ldots \leq d_n$. Assuming the contrary,
we get $d_{s+k}\geqslant d_s+k$ for all
$s=1,2,\ldots,n-k$. Let each of the vertices $k+1,\ldots,n$
send a postcard to all their neighbors. Then
for each $s=1,\ldots,n-k$
vertex number $s+k$ sends
at least $d_{s+k}-k\ge d_s$ postcards to vertices
with numbers from 1 to $n-k$, and for $s=n-k$ this inequality is strict. But vertices with numbers from 1 to
$n-k$ can receive at most $d_1+\ldots+d_{n-k}$
postcards — a contradiction.
\end{proof}

\newpage


\section{Formulation of the Main Results}

The following theorem provides an upper bound for the quantity $f(n,k)$.
\begin{theorem}\label{main}
There exists a graph
$G_{n,k}$ on $n$ vertices such that $h_k(G_{n,k})=f_0(n,k)$, where
$$
f_0(n,k):=
\left(\left\lceil \frac{n}{k} \right\rceil - 2\right)\binom{k}{2} + \binom{k+1}{2} +
\binom{n-k \left( \left\lceil\frac{n}{k} \right\rceil - 1 \right)-1
}{2}
$$
Therefore, $f(n,k)\leqslant f_0(n,k)$.
\end{theorem}

We conjecture that this bound is sharp, i.e., $f(n,k)=f_0(n,k)$
for all $n>k$:
\begin{hypothesis*}
$h_k(G)\geqslant f_0(n,k)$ for any graph
$G$ on $n$ vertices.
\end{hypothesis*}

\begin{remark*}
For $k<n < 17$, the conjecture holds, as verified by computer enumeration.
\end{remark*}

The following theorems are special cases of the conjecture that have been proven.

\begin{theorem}\label{dva}
The conjecture holds for $k \leq 2$.
\end{theorem}

\begin{theorem}\label{tri}
The conjecture holds for
$ k < n \leq 2k$.
\end{theorem}

\begin{theorem}\label{chetyre}
The conjecture holds if
$ k$ divides $n$ ($k | n$).
\end{theorem}

\begin{theorem}\label{pyat}
The conjecture holds if the remainder of the division of $n$ by $k$
is at least $2k/3$.
\end{theorem}

\newpage


\section{Proofs of the Main Results}

\begin{definition}\label{main} Let $G=(V,E)$ be a graph with $n$ vertices.
Define numbers $i,t$ as follows:
$n=ki+t$, $1 \leq t \leq k$, i.e., $i=\lceil n/k\rceil-1$, $t=n-k(\lceil n/k\rceil-1)$.
An unordered pair of vertices $(x,y)$ of the graph $G$ is called \emph{close} if $|\deg(x)-\deg(y)| < k$.
Denote by $A(G)$ the set of close pairs:
$A(G)=\{(x,y)|x,y \in V(G), |\deg(x)-\deg(y)| < k\}$.
\end{definition}

\begin{proof}[Proof of Theorem 1.]
Let us construct such a graph. Partition the vertices into $i + 1$ groups
$V_0,\ldots,V_i$ with indices
$0, \dots, i$. Group $V_0$ contains $t-1=n-ki-1$ vertices, group $V_{\lfloor (i+1)/2\rfloor }$ contains $k + 1$
vertices, and the others contain $k$ vertices each (the total number of vertices is
$n-ki-1+ki+1=n$, as required.)

An edge between two vertices (which may lie in
the same group or in different groups) is drawn
if and only if the sum of the indices of their groups is strictly greater than $i$. In particular,
vertices in group $V_0$ have degree 0.
Let us verify that this graph works. Take a vertex $v$ from group $V_j$,
$j\in \{1,\ldots,i\}$, and compute its
degree. Vertex $v$ is connected to all vertices from the $j$ groups
$V_{i-j+1},\ldots,V_i$ and
only to them. Then, to all these groups, except possibly the groups with indices $j$ and
$\lfloor (i+1)/2\rfloor$,
there are exactly $k$ edges (to group $V_j$ there is 1 less edge, because there is no edge
to the vertex $v$ itself; to group $V_{\lfloor (i+1)/2\rfloor}$ there is 1 more
edge, because it contains 1 more vertex; if it is the same group,
these effects cancel out and the number of edges to group $V_j$ is $k$).
Note that among the groups $V_{i-j+1},\ldots,V_i$, group
$V_j$ is present if and only if
$j\geqslant i-j+1 \Leftrightarrow j\geqslant \lceil (i+1)/2\rceil$,
and group $V_{\lfloor (i+1)/2\rfloor}$ is present if and only if
$i-j+1\leqslant \lfloor (i+1)/2\rfloor \Leftrightarrow j\geqslant (i+1)-
\lfloor (i+1)/2\rfloor=\lceil (i+1)/2\rceil$ — so either both of these special groups are present, or neither is. Thus,
in total, from our vertex $v$,
there are exactly $kj$ edges to the groups $V_{i-j+1},\ldots,V_i$. Therefore, the pairs of vertices of the graph $G$
whose degrees differ by less than $k$ are precisely the pairs
of vertices from the same group. Obviously, the number of such pairs is exactly $f_0(n,k)$.
\end{proof}

\begin{proof}[Proof of Theorem 2.]
For $k=1$ we have $f_0(n,1)=1$, i.e.,
in this case we need to prove that
the graph contains two vertices of equal
degree. This is well known and almost obvious: if
the degrees of all vertices of the graph $G$ are distinct, then they take
all values from 0 to $n-1$ exactly once, but this is impossible because a vertex of degree $n-1$ must be adjacent to all others, while a vertex of degree 0 is adjacent to none.

2) Let $k=2$, then
$f_0(n,2)= \lceil n/2\rceil + 1$.

Suppose the contrary: among the pairs of vertices of the graph $G$ there are at most $\lceil n/2\rceil$
close pairs. Label the vertices of the graph $G$ as $v_1,\ldots,v_{n}$ in order of
increasing degrees:
$d_1\leqslant d_2\leqslant \ldots \leqslant d_{n}$, where
$d_i=\deg(v_i)$.
Consider the $n-1$ pairs of vertices $(v_i,v_{i+1})$, call such pairs \emph{consecutive}.

Consider two cases.

2.1) Let $n=2m\geqslant 4$ be even, then $\lceil n/2\rceil=m$.

Suppose there are at most $m-1$ close pairs among the consecutive pairs, then there are at least $m$ non-close pairs.
Then $$d_{2m}-d_{1}=\sum_{i=1}^{2m-1}(d_{i+1}-d_i)\geqslant 2\cdot m,$$
a contradiction.

Suppose there are exactly $m$ close consecutive pairs
(and these are all close pairs). Arguing similarly, we get
that $d_{2m}-d_1\geqslant 2m-2=n-2$. The reasoning from item 1 shows
that this is the largest possible value for the difference of vertex degrees in a graph
on $n$ vertices,
so equality must hold everywhere. Thus, in all close consecutive pairs
the degrees are equal, and in all non-close pairs they differ by exactly 2. Furthermore,
two close consecutive pairs cannot occur consecutively: otherwise, there would be 3 vertices of the same degree and, consequently, a close non-consecutive pair. There are
$2m-1$ consecutive pairs in total, $m$ of which are close, meaning the close and non-close consecutive
pairs alternate, with the first and last consecutive pairs being close.
It follows that either $d_1=d_2=0$, $d_{2m}=2m-2$ (which is impossible: vertex
$v_{2m}$ can only be adjacent to vertices $v_3,\ldots,v_{2m-1}$, so
$d_{2m}\leqslant 2m-3$), or
$d_1=1$, $d_{2m}=d_{2m-1}=2m-1$ (which is also impossible: vertices
$v_{2m},v_{2m-1}$ are adjacent to everyone, so
$d_{1}\geqslant 2$). Contradiction.

2.2) Let $n=2m+1\geqslant 3$ be odd. Then $\lceil n/2\rceil=m+1$.
Suppose among the $2m$ consecutive pairs there are at most $m$ close ones, then
there are at least $m$ non-close ones and $d_{2m+1}-d_1\geqslant 2m$, a contradiction
(similar to item 1). Then there are exactly $m+1$ close consecutive pairs (and these are all close pairs), so among them there are two consecutive ones:
$d_{j+1}-d_j\leqslant 1$ and $d_{j+2}-d_{j+1}\leqslant 1$ for some
$j$. Since the pair $(v_j,v_{j+2})$ cannot be close,
we get $d_{j+1}-d_j=d_{j+2}-d_{j+1}=1$. Hence
$d_{2m+1}-d_1\geqslant 2\cdot 1+(m-1)\cdot 2=2m$, a contradiction,
similar to item 1.
\end{proof}

\begin{proof}[Proof of Theorem 3.]
Let $n=k+t$, $0<t\leqslant k$, then $f_0(n,k)=\binom{k+1}{2}+\binom{t-1}{2}$.
By Lemma 1, there exist $k+1$ pairwise close vertices in $G$.
Let the degrees of the vertices in this set lie between $x$ and $x+k-1$, where
$0\leqslant x\leqslant x+k-1\leqslant n-1=t+k-1$.
Then $0\leqslant x\leqslant t\leqslant k\leqslant x+k\leqslant n$.
Denote by $a,b,c,d,e$ respectively the number of vertices of the graph $G$
whose degrees belong to the half-intervals $[0,x)$, $[x,t)$,
$[t,k)$, $[k,x+k)$ and $[x+k,n)$, so $a+b+c+d+e=n=k+t$. Fix that $b+c+d \geq k+1$ (and hence $a+e \leq t-1$). Note that at least one
of the inequalities $e<b+c$, $a<c+d$ holds, since otherwise $t-1\geqslant a+e\geqslant b+2c+d\geqslant b+c+d\geqslant k+1$, which is false. Therefore
$ae\leqslant \max(a(b+c),e(c+d))\leqslant a(b+c)+e(c+d)$.
The number of non-close pairs of vertices of the graph $G$ does not exceed
$a(d+e)+be\leqslant ad+be+a(b+c)+e(c+d)=(a+e)(b+c+d)\leqslant (k+1)(t-1)$
(since the product of two numbers with sum $n$, one of which is at least $k+1>n/2$,
does not exceed $(k+1)(n-k-1)=(k+1)(t-1)$.) Then the number of close
pairs of vertices of $G$ is at least
$$
\binom{n}{2}-(k+1)(t-1)=\binom{k+1}{2}+\binom{t-1}{2}=f_0(n,k),
$$
as required.
\end{proof}

\begin{proof}[Proof of Theorem 4.] Divide the set $\{0, ..., n - 1\}$ into $n/k$
intervals of size $k$.
Let $V_j$ ($j=1,\ldots,n/k$) be the set of vertices of $G$
whose degree belongs to the $j$-th interval, $x_j=|V_j|$.
Since the degrees of any two vertices from $V_j$
differ by at most $k-1$, this already gives the estimate
$h_k(G)\geqslant \sum \binom{x_j}{2}\geqslant (n/k)\binom{k}{2}$
(by Jensen's inequality for the function $\binom{x}{2}$).
Note that
$f_0(n,k)=(n/k-2)\binom{k}{2}+\binom{k+1}{2}+\binom{k-1}{2}=
(n/k)\binom{k}{2}+1$. Thus, if $h_k(G)<f_0(n,k)$, then each $V_i$ contains exactly
$k$ vertices and for any pair of vertices
from $V_i$ and $V_j$ with $i\ne j$, their
degrees differ by at least $k$.
Note that this immediately implies that
$\deg u-\deg v\geqslant (j-i)k$ for $v\in V_i$, $u\in V_j$, where
$j>i$.
Let us analyze the cases.

Suppose $(2k)$ divides $n$, i.e., $n/k$ is even.
Partition the vertices of $G$ into 2 equal groups
(with degree $>(n-1)/2$ and with degree
$<(n-1)/2$). Assign to the vertices of the second group
vertices from the first group such that vertices in $V_i$ correspond to
vertices in $V_{i+n/(2k)}$, $i=1,\ldots,n/(2k)$.
In each pair, the degree of the vertex from the first group $u$ exceeds the degree of the corresponding
vertex from the second group $v$ by at least
$k\cdot (n/2k)=n/2$. Consequently, from $u$
there are strictly more edges to the second group than from $v$ to
the first group. Summing this over all $n/2$ pairs
we get that $E>E$, where $E$ is the number of edges between
the first and second groups. Contradiction.

Now let us consider the case when $n/k=(2s+1)$ is odd.
Let the first group be the union of sets $V_{s+2},\ldots,V_{2s+1}$,
the second group be the union of sets $V_{1},\ldots,V_{s+1}$.
Let $v_1,\ldots,v_k$ be the vertices
of set $V_{s+1}$; $u_1,\ldots,u_k$ be the
vertices of set $V_1$;
$w_1,\ldots,w_k$ be the
vertices of set $V_{2s+1}$.
Then $\deg u_i\leqslant \deg v_i-ks$,
so from the set $\{u_i,v_i\}$ to the first group there are
at most $\deg u_i+ks\leqslant \deg v_i$ edges.
On the other hand, from $w_i$ to the second group there are
at least $\deg w_i-(ks-1)\geqslant \deg v_i+1$
edges. To the remaining vertices of the second group, assign the remaining
vertices of the first group as in the previous case: to a vertex
from $V_i$ assign a vertex from $V_{i+s}$
($i=2,3,\ldots,s$). In each pair, the vertex $u$ from the first group
has a degree at least $ks$ greater than the vertex $v$ from
the second group,
so more edges go from it
to the first group than from $u$ to the second. Summing up, we again get
$E>E$, where $E$ is the number of edges between
the first and second groups. Contradiction.
\end{proof}

We preface the proof of Theorem 5 with a series of lemmas.
Using Theorems 2, 3, and 4, we assume that $k\geqslant 3$, $n\geqslant 2k+1$ and $k>t$.

\begin{lemma}\label{l2}
Let graph $G$ be a counterexample to the conjecture for some $n>k$,
$n=ki+t$, $1\leqslant t\leqslant k$.
Partition the set of numbers $\{0, 1, \dots, n-1\}$ into $i+1$ intervals (numbered from 0 to $i$)
so that the size of each interval is at most $k$. Then for any interval
of the partition, the graph $G$ contains at least $t$ vertices with a degree in that interval.
\end{lemma}
\begin{proof}
Let $x_j$ be the number of vertices of $G$ with a degree
in the $j$-th interval ($j=0,\ldots,i$). Then $h_k(G)\geqslant \sum \binom{x_j}{2}$,
while $f_0(n,k)=\binom{t-1}{2}+\binom{k+1}{2}+(i-1)\binom{k}{2}=\sum \binom{y_j}{2}$,
where $y_0=t-1$, $y_1=k+1$, $y_2=\ldots=y_i=k$.
If the statement of the lemma is false, then $\min_j x_j\leqslant t-1$. By equalizing
the maximum and minimum of the numbers $x_1,\ldots,x_i$ while preserving
the sum, we can achieve that one of them becomes equal to $t-1$, and the sum
$\sum \binom{x_j}{2}$ can only decrease. But when
one of the $x_j$ is equal to $t-1$, the sum of the others is
$(k+1)+(i-1)k$, so the sum of the values of the convex function $\binom{x}{2}$
over the others is at least $\binom{k+1}{2}+(i-1)\binom{k}{2}$
(this can also be proven, for example, by equalizations preserving the sum). Hence
$\sum_j \binom{x_j}{2}\geqslant f_0(n,k)$ and $G$ is not a counterexample, a contradiction.
\end{proof}

Further, groups of vertices of the graph will be defined as in the previous lemma.

The following lemma holds for any fixed disjoint partition of the numbers $\{0, 1, ... n-1\}$ into $i+1$ intervals of size at most $k$. We will henceforth number the groups in order of increasing vertex degrees in them — from the zeroth to the $i$-th.

\begin{lemma}\label{l3}
Assume that in each group there are at least $t+p_2$ vertices for some integer $p_2 \geq 0$.
Let $x$ be the average degree of vertices in the zeroth group. Then the average degree of vertices in the $i$-th group is at most $ki+x-p_2-1$.
\end{lemma}
\begin{proof}
If the claim is false, then the average degree in the $i$-th group is greater than $ki+x-p_2-1$,
then the average non-degree (number of non-adjacencies) in the $i$-th group is less than $t-x+p_2$. Let
$y,z$ be the numbers of vertices in the zeroth and $i$-th groups, respectively.
Then between the zeroth and $i$-th groups there are at most $xy$
edges and fewer than $(t-x+p_2)z$ non-edges. Then
    $yz < xy + (t-x+p_2)z\le (t+p_2)\max(y,z)$, whence $\min(y,z)< t+p_2$ ---
    a contradiction.
\end{proof}

Let us highlight the inequalities used in what follows.

\begin{lemma}\label{l5}
    Let $i\geqslant 2$ be a natural number,
    $a_1,\ldots,a_{i}$ and
    $b_0,\ldots,b_{i-1}$ be nonnegative numbers.
    Also set
    $a_0=b_i:=0$.
    Then $$\sum_{j=0}^{i-1} b_ja_{j+1}\geqslant \min\{(b_j+a_j)
    (b_{j+1}+a_{j+1}),
    0\le j\le i-1\}.$$
\end{lemma}

\begin{proof}
    Proof by induction on $i$.

    Base case $i=2$: $b_0a_1+b_1a_2\geqslant \min(b_0,a_2)(a_1+b_1) = \min\{b_0(a_1+b_1), (b_1+a_1)a_2\}$.

    Inductive step from $i-1$ to $i\ge 3$. Using the induction hypothesis,
    we reduce the problem to verifying the inequality
    \begin{align*}
        &\min\{b_0(b_1+a_1),\ldots,
    (b_{i-2}+a_{i-2})a_{i-1}\}+b_{i-1}a_i \ge \\
    &\min\{b_0(b_1+a_1),\ldots,
    (b_{i-1}+a_{i-1})a_{i}\}
    \end{align*}
    \\
    Either the minimum on the left-hand side is not less than that on the right-hand side (and then the inequality holds), or the left-hand side is equal to
    \begin{align*}
        (b_{i-2}+a_{i-2})\cdot a_{i-1} + b_{i-1}\cdot a_i\geqslant (b_{i-1}+a_{i-1})\min\{a_i,b_{i-2}+a_{i-2}\}\\
        = \min\{(b_{i-1}+a_{i-1})a_i, (b_{i-1}+a_{i-1})(b_{i-2}+a_{i-2})\} \geq \min\{b_0(a_1+b_1), \dots ,(b_{i-1}+a_{i-1})a_i\}
    \end{align*}
\end{proof}

\begin{lemma}\label{l35}
Let $m$ be a natural number,
$x_1,\ldots,x_m$ be integers,
$\theta(x)$ be a function of an integer argument,
defined on the interval $[\min(x_1,\ldots,x_m),\max(x_1,\ldots,x_m)]$ and convex on it, i.e.,
the difference
$\theta(x)-\theta(x-1)$ is increasing.
Suppose that integers $A,B,k$
are such that $A+B=m$, $Ak+B(k+1)=\sum x_i$. Then $\sum \theta(x_i)\ge A\cdot \theta(k)+B\cdot \theta(k+1)$.
\end{lemma}

\begin{proof}   Consider two cases.

1) Both numbers $A,B$ are nonnegative.
Then by shifting the integers $x_1,\ldots,x_{m}$ while preserving
    their sum, we can achieve that any two of them differ by at most 1
    (if some two differ by more than 1, then bring them closer to each other while preserving the sum. The process terminates because the sum of the numbers does not change, and the sum of squares of the numbers strictly decreases).
    The left-hand side of our inequality does not increase in this process.
    In the end, any two numbers will differ by at most 1, and then $A$ of them will be equal to $k$,
    and the remaining $B$ will be equal to $k+1$ — so the left-hand side becomes equal to the right-hand side. Since the left-hand side did not increase during the process, it was originally at least the right-hand side.

    2) One of the numbers $A$, $B$ is negative (for example, $B$, the second case is similar). Then rewrite the inequality as
    $\sum \theta(x_i)+(-B)\theta(k+1)\geqslant (m-B)\theta(k)$
    and apply the statement proven in item 1) to the numbers
    $x_1,\ldots,x_m$, $(-B)$ times $k+1$, which are $m-B$ in total and their sum is $(m-B)k$.
\end{proof}

\begin{lemma}\label{l36}
    Suppose $n=ki+t$ vertices are partitioned into $i+1$ groups
    (as before, the partition is induced by the partition
    of the set of possible degrees $\{0,1,\ldots,n-1\}$
    into $i+1$ intervals). Let the size of the smallest group
    be $t+p_2$, and the size of the smallest of the remaining groups be $t+p_1$, where $p_1, p_2\geqslant 0$.
    If $p_1+p_2 \leq k-t$, then the number of pairs of vertices belonging to the same
    group is not less than
$$
f_0(n,k) - (-k^2/2+k(p_1+p_2)+kt+k/2-p_1^2/2-p_1t+p_1/2-p_2^2/2-p_2t+p_2/2-t^2/2-t/2+1).
$$
    If $p_1+p_2 \geq k-t$, then the number of pairs of vertices belonging to the same
    group is not less than
$$
f_0(n,k) - (-k^2/2+k(p_1+p_2)+kt+3k/2-p_1^2/2-p_1t-p_1/2-p_2^2/2-p_2t-p_2/2-t^2/2-3t/2+1).
$$
\end{lemma}

\begin{proof} We apply
Lemma  \ref{l35} to the function ${x\choose 2}$, where the numbers
$x_j$, $j=1,\ldots,i-1$ are the sizes of all groups except the two
smallest.
If $p_1+p_2\le k-t$, we obtain that
$$
\sum {x_j\choose 2}\geqslant (i-1-k+t+p_1+p_2){k\choose 2}+(k-t-p_1-p_2){k+1\choose 2}.
$$
 If $p_1+p_2 \geq k-t$, then
$$
\sum {x_j\choose 2}\geqslant (i-1+k-t-p_1-p_2){k\choose 2}+(-k+t+p_1+p_2){k-1\choose 2}.
$$

To complete the proof of the lemma, it suffices to add
${t+p_1\choose 2}+{t+p_2\choose 2}$ to
these expressions and
expand the brackets.
\end{proof}

Pairs of close vertices from one group are not enough for us. The following lemma guarantees the existence of a sufficiently large
number of pairs of close vertices from different groups.

\begin{lemma}\label{l4}
Partition the numbers from 0 to $n-1$ into $i+1$ intervals
in the following way: the first $i$ intervals contain $k$ numbers each,
the last one --- $t$ numbers. Accordingly,
the set of vertices of the graph $G$ is partitioned into $i+1$ groups.
Let $t+p_1$ and $t+p_2$, $p_1 \geq p_2 \geqslant 0$ be the minimal sizes of the groups.
Then in such a $G$ there are at least $\frac{(t+p_1)(t+p_2)(p_2+1)}{2k-p_2-3}$ close pairs of vertices from different groups. 
\end{lemma}
\begin{proof}
Let $x$ be the average degree of vertices in the zeroth group. Let $y\leq ki+x-p_2-1$ be the average degree of vertices in the $i$-th group (we used Lemma \ref{l3}).
Since in the $i$-th group all degrees are at least $n-t=ki$,
we get that $x\geq p_2+1$.

Let in the $j$-th group,
$j=0,1,\ldots,i$, there be exactly $b_j$ vertices
whose degree is greater than
$\frac{x+y-ki}{2}+kj$, and $a_j$ vertices
whose degree is less than or equal to
$\frac{x+y-ki}{2}+kj$. Note that
the sought number of close pairs of vertices from
different groups is not less than $b_0a_1+b_1a_2+\ldots+b_{i-1}a_i\ge \min\{b_0(a_1+b_1), (a_1+b_1)(a_2+b_2), (a_2+b_2)(a_3+b_3), \dots , (a_{i-2}+b_{i-2})(a_{i-1}+b_{i-1}), (a_{i-1}+b_{i-1})a_i\}$ by Lemma
\ref{l5}.

Let us estimate $b_0$ from below.
Let the proportion of vertices with degree greater than
$\frac{x+y-ki}{2}$
in the zeroth group be $r$, set
$q=1-r$. Then the average degree
$x$ of the zeroth group is not greater than
\begin{align*}
x&\leq r\cdot (k-1)+q\cdot \left(\frac{x+y-ki}{2}
\right)\leqslant  r(k-1)+q\left(x-\frac{p_2+1}{2}\right)\\
&=x-\frac{p_2+1}2+r\left({k + \frac{p_2+1}{2} - x - 1}\right)\\
&\leq  x-\frac{p_2+1}2+r\left({k -\frac{p_2+1}{2}- 1}\right),
\end{align*}
whence $2k-p_2-3>0$ and $1\geqslant r\geqslant \frac{p_2+1}{2k-p_2-3}$ (the left inequality follows from the definition of $r$).

Now let us estimate $a_i$ from below.
Let the proportion of vertices with degree
$\leqslant \frac{x+y+ki}{2}$
in the $i$-th group be $r'$, set $q'=1-r'$.
Then the average degree $y$ in the $i$-th group is not less than
\begin{align*}
   y&\geq  r'\cdot ki+q'\frac{x+y+ki}{2}
   \geq r'\cdot ki+q'\left(y+\frac{p_2+1}{2}\right)
    =r'\left(ki - \frac{p_2+1}{2} - y\right) + y + \frac{p_2+1}{2},
\end{align*}
whence
$$
\frac{p_2+1}{2}\leq r'\left(\frac{p_2+1}{2}+y-ki\right)\leq
r'\left(x -\frac{p_2+1}{2}\right)
\leq r'\left({k -\frac{p_2+1}{2}- 1}\right),
$$
that is, $2k-p_2-3>0$ and $1\geqslant r'\geqslant \frac{p_2+1}{2k-p_2-3}$ (the left inequality follows from the definition of $r'$).

Then $\min\{b_0(a_1+b_1), (a_1+b_1)(a_2+b_2), (a_2+b_2)(a_3+b_3), \dots , (a_{i-2}+b_{i-2})(a_{i-1}+b_{i-1}), (a_{i-1}+b_{i-1})a_i\} \geqslant \min\{\frac{p_2+1}{2k-p_2-3}(a_0+b_0)(a_1+b_1), (a_1+b_1)(a_2+b_2), (a_2+b_2)(a_3+b_3), \dots , (a_{i-2}+b_{i-2})(a_{i-1}+b_{i-1}), (a_{i-1}+b_{i-1})(a_i+b_i)\frac{p_2+1}{2k-p_2-3}\}$

We know that $a_j+b_j\geqslant t+p_2$
for all $j=1,2,\ldots,i-1$ and $a_j+b_j\geqslant t+p_1$
for all $j=1,2,\ldots,i-1$, except one. Hence, $\min\{\frac{p_2+1}{2k-p_2-3}(a_0+b_0)(a_1+b_1), (a_1+b_1)(a_2+b_2), (a_2+b_2)(a_3+b_3), \dots , (a_{i-2}+b_{i-2})(a_{i-1}+b_{i-1}), (a_{i-1}+b_{i-1})(a_i+b_i)\frac{p_2+1}{2k-p_2-3}\} \geqslant \frac{(t+p_1)(t+p_2)(p_2+1)}{2k-p_2-3}$.
\end{proof}

\begin{remark*}
    $i(t+p_1)+t+p_2 \leq ki+t \Rightarrow{} p_1 \leq k-t$
\end{remark*}

\begin{proof}[Proof of Theorem 5.]
Suppose $t \geq 2k/3$. Partition the numbers from 0 to $n-1$
and, accordingly, the vertices of the graph as in the previous lemma.
Denote by $t+p_2$ the size of the smallest group, and by $t+p_1$ the size of the smallest of the remaining groups. Lemma \ref{l2} allows us to assume that $p_2\geq 0$.

Let $p_1+p_2 \geq k-t$.
By Lemma \ref{l36}, it suffices to prove that the number of close pairs of vertices
of the graph $G$ lying in different groups is strictly greater
than $-k^2/2+k(p_1+p_2)+kt+3k/2-p_1^2/2-p_1t-p_1/2-p_2^2/2-p_2t-p_2/2-t^2/2-3t/2$. This number, by Lemma
\ref{l4}, is at least
$\frac{(t+p_1)(t+p_2)(p_2+1)}{2k-p_2-3}$, so the theorem reduces to the inequality:
\begin{align*}
g(k,t,p_1,p_2)=2k^3-4k^2t-4k^2p_1-5k^2p_2-9k^2+2kt^2+4ktp_1+6ktp_2+12kt+2kp_1^2+\\
+2kp_1p_2+8kp_1+4kp_2^2+11kp_2+9k+t^2p_2-t^2-4tp_1-7tp_2-9t-p_1^2p_2-3p_1^2+\\
+2p_1p_2^2+p_1p_2-3p_1-p_2^3-4p_2^2-3p_2 > 0
\end{align*}

We have \begin{align*}
\frac{\partial g}{\partial k}=6k^2-8kp_1-10kp_2-8kt-18k+2t^2+4tp_1+\\
+6tp_2+12t+2p_1^2+2p_1p_2+8p_1+4p_2^2+11p_2+9
\end{align*}
Let us check that this is less than 0. This is
a convex quadratic trinomial in
$k$, so it suffices to check at the extreme points $t+p_1, t+p_1+p_2$. We have
\begin{align*}
\frac{\partial g}{\partial k}(t+p_1+p_2, t, p_1, p_2) = -4p_1p_2-10p_1-7p_2-6t+9,
\end{align*} this is
negative for $t \geq 2$. Next,
\begin{align*}
\frac{\partial g}{\partial k}(t+p_1, t, p_1, p_2) = -8p_1p_2-10p_1+4p_2^2-4p_2t+11p_2-6t+9 \leq -3p_1p_2-4p_2t-6t+9,
\end{align*}
which is negative for $t \geq 2$. If $t=1$, then $k \leq 3/2$, and this case has already been dealt with. Hence, it suffices to check
\begin{align*}
\frac{1}{2}g(t+p_1+p_2,t,p_1,p_2) = 3p_1-2p_1^2+3p_2+p_1p_2-2p_1^2p_2-p_2^2-p_1t-p_2t-p_1p_2t+\\
+p_2^2t+t^2+p_2t^2 > 0.
\end{align*}
Note that if $p_1 = 0$, then $p_2=0$ and $\frac{1}{2}g(t+p_1+p_2,t,p_1,p_2) = t^2 > 0$.\\
Let $t = x(p_1+p_2)$. Then $\frac{2}{3} \leq \frac{t}{k} = \frac{x}{x+1} \Rightarrow{} 2 \leq x$
\begin{align*}
\frac{1}{2}g(t+p_1+p_2,t,p_1,p_2) = (p_1^2p_2+p_1^2+p_2^2)(x^2-x-2)+\\
+p_2^3(x^2+x)+2p_1p_2^2x^2+p_1p_2(2x^2-2x+1)+3p_1+3p_2,
\end{align*}
all terms are nonnegative for $x \geq 2$, and some are positive for $p_1 > 0$.

Let us proceed to the second case. Let $p_1+p_2 \leq k-t$.
By Lemma \ref{l36}, it suffices to prove that the number of close pairs of vertices
of the graph $G$ lying in different groups is strictly greater
than $-k^2/2+k(p_1+p_2)+kt+k/2-p_1^2/2-p_1t+p_1/2-p_2^2/2-p_2t+p_2/2-t^2/2-t/2$. This number, by Lemma
\ref{l4}, is at least
$\frac{(t+p_1)(t+p_2)(p_2+1)}{2k-p_2-3}$, so the theorem reduces to the inequality:
\begin{align*}
2k^3-4k^2t-4k^2p_1-5k^2p_2-5k^2+2kt^2+4ktp_1+6ktp_2+8kt+2kp_1^2+2kp_1p_2+\\
+4kp_1+4kp_2^2+5kp_2+3k+t^2p_2-t^2-4tp_1-5tp_2-3t-p_1^2p_2-3p_1^2+\\
+2p_1p_2^2+3p_1p_2+3p_1-p_2^3-2p_2^2+3p_2 = \\
= 2k(k-(t+p_1+p_2))^2+4t(k-(t+p_1+p_2))+p_2^2(p_1-p_2)+ \\
+3(k-t)+2p_2(p_1-p_2)+p_1p_2+3p_1+3p_2+kp_2(t-2p_1)+ \\
+ktp_2+2kp_2^2+4kt+3t^2+p_1p_2^2+4kp_1+5kp_2+t^2p_2-k^2p_2-tp_2-5k^2-p_1^2p_2-3p_1^2> 0
\end{align*}
Let us simplify:
\begin{align*}
p_1 \geq p_2 \geq 0\\
t \geq 2k/3 \Rightarrow{} t^2p_2+ \frac{5}{6}ktp_2 \geq k^2p_2 \\
p_1 \leq k-t \leq t/2, p_1 \leq k-t \leq k/3 \Rightarrow{} \frac{1}{6}ktp_2 \geq p_1^2p_2 \\
t \geq 2k/3 \Rightarrow{} \frac{9}{4}t^2 \geq k^2, \\
kp_2(t-2p_1) \geq 0 \\
p_1 \leq  t/2 \Rightarrow{} \frac{3}{4}t^2 \geq 3p_1^2 \\
kp_2 \geq tp_2
\end{align*}
Applying these 7 inequalities, we reduce the large inequality to
\begin{align*}
2k(k-(t+p_1+p_2))^2+4t(k-(t+p_1+p_2))+3(k-t)+4kt+4kp_1+4kp_2-4k^2 = \\
2k(k-(t+p_1+p_2))^2+4t(k-(t+p_1+p_2))+3(k-t)+4k((t+p_1+p_2)-k) = \\
(k-(t+p_1+p_2))(2k(k-(t+p_1+p_2))-2k)+(k-(t+p_1+p_2))(4t-2k)+3(k-t)> 0
\end{align*}
If $k-(t+p_1+p_2) \geq 1$, all terms except the last are clearly nonnegative. If $k-(t+p_1+p_2) = 0$, they vanish. Therefore, the inequality reduces to $k-t>0$.
\end{proof}


\end{document}